\documentclass[12pt]{article}
\usepackage{amsmath,amsfonts,amssymb,amsthm}
\newtheorem{theo}{Theorem}
\newtheorem{lem}{Lemma}
\newtheorem*{conj}{Conjecture}

\title{Two integer sequences related to Catalan numbers}
\author{Michel Lassalle\\
\small Centre National de la Recherche Scientifique\\[-0.8ex]
\small Institut Gaspard-Monge, Universit\'e de Marne-la-Vall\'ee\\[-0.8ex]
\small 77454 Marne-la-Vall\'ee Cedex, France\\[-0.8ex]
\small \texttt{lassalle@univ-mlv.fr}\\[-0.8ex]
\small \texttt{http://igm.univ-mlv.fr/{\textasciitilde}lassalle}
}
\date{}
\begin{document}
\maketitle

\begin{abstract}
We prove the following conjecture of Zeilberger. Denoting by $C_n$ the Catalan number, define inductively $A_n$ by $(-1)^{n-1}A_n=C_n+\sum_{j=1}^{n-1} (-1)^{j} \binom{2n-1}{2j-1} A_j \,C_{n-j}$ and $a_n=2A_n/C_n$. Then $a_n$ (hence $A_n$) is a positive integer.
\end{abstract}

\section{Introduction}\label{secone}

Let $z$ be an indeterminate. For any nonnegative integer $r$, we denote by $C_r(z)$ the Narayana polynomial defined by $C_0(z)=1$ and
\[C_r(z)=\sum_{k=1}^{r} N(r,k) z^{k-1},\]
with the Narayana numbers $N(r,k)$ given by
\[ N(r,k)=\frac{1}{r} \binom{r}{k-1} \binom{r}{k}.\]
There is a rich combinatorial litterature on this subject. Here we shall only refer to~\cite{MS,Su},~\cite[Exercice 36]{S} and references therein. 

We have $C_r\eqref{eqone}=C_r$, the ordinary Catalan number, since
\[C_r=\frac{1}{r+1}\binom{2r}{r}=\sum_{k=1}^{r} N(r,k).\]
Moreover~\cite{C} the Narayana polynomial $C_r(z)$ can be expressed in terms of Catalan numbers as
\begin{equation}\label{eqone}
C_r(z)=\sum_{m\ge0} z^m (z+1)^{r-2m-1} \binom{r-1}{2m}C_m.
\end{equation}

The polynomial 
\[(z+1)\,C_r(z)-C_{r+1}(z)=(1-r)z^{r-1}+\ldots\]
can be expressed, in a unique way, in terms of the monic polynomials $z^mC_{r-2m+1}(z)$ with degree $r-m$. As a consequence, 
we may define the real numbers $A_m(r)$ by
\begin{equation}\label{eqtwo}
(z+1)\,C_r(z)-C_{r+1}(z)=\sum_{m\ge1}
(-z)^m\binom{r-1}{2m-1}A_m(r)\, C_{r-2m+1}(z).
\end{equation}
Note that the real numbers $\{A_m(r), m\ge 1\}$ depend on $r$ and that we have $A_1(r)=1$.

By relation \eqref{eqone}, the left-hand side may be written as
\[-\sum_{k\ge0} z^k (z+1)^{r-2k} \binom{r-1}{2k-1}C_k.\]
Similarly the right-hand side becomes
\[\sum_{m\ge1, n\ge 0}
(-1)^m z^{m+n} (z+1)^{r-2m-2n} \binom{r-1}{2m-1}
\binom{r-2m}{2n}A_m(r)\,C_n.\]
Comparing coefficients of $z^k(z+1)^{r-2k}$ on both sides yields
\[-C_{k}\binom{r-1}{2k-1}=\sum_{m+n=k}(-1)^m\binom{r-1}{2m-1}\binom{r-2m}{2n}A_m(r)\,C_n,\]
with $r\ge2k$. This is easily transformed to
\[C_n=\sum_{j=1}^n (-1)^{j-1} \binom{2n-1}{2j-1} A_j(r) \,C_{n-j}.\]
In other words, $A_n:=A_n(r)$ is an integer independent of $r$, given by the recurrence formula
\begin{equation}\label{eqthree}
(-1)^{n-1}A_n=C_n+\sum_{j=1}^{n-1} (-1)^{j} \binom{2n-1}{2j-1} A_j \,C_{n-j}.
\end{equation}  
This relation can be chosen as a definition, equivalent to \eqref{eqtwo}.

In Section \ref{sectwo} we prove
\begin{theo}\label{thone}
The integers $\{A_n, n\ge 2\}$ are positive and increasing.
\end{theo}
The values of $A_n$ for $1\le n\le 14$ are given by
\begin{align*}
1, 1, 5, 56, 1092, 32670, 1387815, 79389310, 5882844968&, 548129834616, 62720089624920,\\
8646340208462880, 1413380381699497200&, 270316008395632253340.
\end{align*}

At the time the positivity of $A_n$ was only conjectured, we asked Doron Zeilberger for an advice. He suggested to consider also the numbers $a_n=2A_n/C_n$ which, by an easy transformation of \eqref{eqthree}, are inductively defined by
\[(-1)^{n-1}a_n=2+\sum_{j=1}^{n-1}(-1)^{j} \binom{n-1}{j-1}\binom{n+1}{j+1}\frac{a_j}{n-j+1}. \]
In Section \ref{secthree} we prove the following conjecture of Zeilberger.
\begin{theo}\label{thtwo}
The numbers $\{a_n, n\ge 2\}$ are increasing positive integers. They are odd if and only if $n=2^k-2$ for some $k\ge 2$.
\end{theo}
The values of $a_n$ for $1\le n\le 16$ are given by
\begin{align*}
& 2, 1, 2, 8, 52, 495, 6470, 111034, 2419928, 65269092, 2133844440,
83133090480, \\ & 3805035352536, 202147745618247, 12336516593999598,
857054350280418290.
\end{align*}
Both sequences seem to be new.

Of course since $2A_n=a_nC_n$, Theorem \ref{thone} is an obvious consequence of Theorem \ref{thtwo}. However we begin by a direct proof of Theorem \ref{thone}, which has its own interest. This proof describes the algebraic framework of our method (the theory of symmetric functions), and it yields a generating function for the $A_n$'s.

The $A_n$'s are also worth a separate study in view of their connection with probability theory, which was recently noticed. Actually Novak~\cite{N} observed, as an empirical evidence, that the integers $(-1)^{n-1}A_n$ are precisely the (classical) cumulants of a standard semicircular random variable. Independently Josuat-Verg\`es~\cite{JV} defined a $q$-semicircular law, which specializes to the standard semicircular law at $q=0$. He described combinatorial properties of its (classical) cumulants, which are polynomials in $q$ having $(-1)^{n-1}A_n$ as their constant terms. It follows from his work that $A_n$ may be obtained by a weighted enumeration of connected matchings (equivalently, fixed-point free involutions on $1,...,2n$ such that no proper interval is stable). Finding a similar bijective proof for $a_n$ would be very interesting.

\section{Properties of $A_n$}\label{sectwo}

This section is devoted to a proof of Theorem \ref{thone}. We shall use two remarks of Krattenthaler and Lascoux. Define
\begin{align*}
\mathbf{C}(z)&=\sum_{n\ge 0} \frac{C_n}{(2n)!}z^n=\sum_{n\ge 0} \frac{1}{n!(n+1)!}z^n,\\
\mathbf{A}(z)&=\sum_{m\ge 1} (-1)^{m-1}\frac{A_m}{(2m-1)!}z^{m-1}.
\end{align*}
Krattenthaler observed that the definition \eqref{eqthree} is equivalent with
\begin{equation}\label{eqfour}
\mathbf{A}(z)\mathbf{C}(z)=2\frac{d}{dz}\mathbf{C}(z).
\end{equation}
Therefore we have only to show that, if we write
\[\mathbf{C}^\prime(z)/\mathbf{C}(z)=
\frac{d}{dz}\mathrm{log}(\mathbf{C}(z))=\sum_{m\ge 1}c_m z^{m-1},\]
the coefficient $c_m$ has sign $(-1)^{m-1}$.

Let $\mathbb{S}$ denote the ring of symmetric functions~\cite[Section 1.2]{Ma}. Consider the classical bases of complete functions $h_n$ and power sums $p_n$. Denote
\[H(z) = \sum_{n\ge 0} h_nz^n, \quad 
P(z)=\sum_{n\ge 1} p_nz^{n-1}\]
their generating functions. It is well known~\cite[p. 23]{Ma} that $P(z)= H^\prime(z)/H(z)$.

Since the complete symmetric functions $h_n$ are algebraically independent, they may be specialized in any way. More precisely, for any sequence of numbers $\{c_n, n\ge 0\}$ with $c_0=1$, there is a homomorphism from $\mathbb{S}$ into the ring of real numbers, 
taking $h_n$ into $c_n$. Under the extension of this ring homomorphism to formal power series, the image of $H(z)$ is the generating function of the $c_n$'s. By abuse of notation, we write $h_n=c_n$ and $H(z)=\sum_{n\ge0} c_nz^n$.

Therefore it is sufficient to prove the following statement.
\begin{lem}\label{lemone}
When the complete symmetric functions are specialized to \[h_n=\frac{1}{n!(n+1)!},\]
i.e. when $H(z)=\mathbf{C}(z)$, the coefficients of $P(-z)$ are all positive.
\end{lem}

Lascoux observed an empirical evidence for the following more general result (which gives Lemma \ref{lemone} when $x=2$). 
\begin{lem}\label{lemtwo}
Let $x$ be a positive real number and $(x)_k = \prod_{i=1}^k (x+i-1)$ denote the classical rising factorial. When the complete symmetric functions are specialized to 
\[h_n=\frac{1}{n!(x)_n},\]
the coefficients of $P(-z)$ are all positive.
\end{lem}
\begin{proof}With this specialization we have
\[H(z)=\sum_{n\ge 0}\frac{1}{(x)_n}\frac{z^n}{n!}={}_0F_{1}(x;z),\]
the so called confluent hypergeometric limit function. Since 
\[\frac{d}{dz}\,{}_0F_{1}(x;z)=\frac{1}{x} \,{}_0F_{1}(x+1;z),\]
we have 
\begin{equation*}
P(z)=\frac{{}_0F_{1}(x+1;z)}{x\,{}_0F_{1}(x;z)}.
\end{equation*}
The classical Gauss's continued fraction~\cite[p. 347]{W} gives the following expression for the right-hand side
\[\frac{\,_0F_1(x+1;z)}{x\,_0F_1(x;z)} = \cfrac{1}{x + \cfrac{z}
{(x+1) + \cfrac{z}{(x+2) + \cfrac{z}{(x+3) + {}\ddots}}}}.\]
This continued fraction may be written as a Taylor series by iterating the usual binomial formula
\[{(1+f_1z)}^{-k_1}=\sum_{k_2\ge0} {f_1}^{k_2} \binom{k_1+k_2-1}{k_2} (-z)^{k_2},\]
where $f_1^{k_2}$ is itself of the form $(c_1(1+f_2z))^{-k_2}$.
This classical method (~\cite[Exercise 4.2]{Las},~\cite{R}) yields
\begin{equation*}
P(z)=\frac{1}{x}\sum_{(k_1,k_2,k_3,\ldots)} \prod_{i\ge 1}\left(\frac{-z}{(x+i-1)(x+i)}\right)^{k_i}\binom{k_i+k_{i+1}-1}{k_{i+1}},
\end{equation*}
where the $k_i$'s are nonnegative integers. This proves the lemma.
\end{proof}

In the previous formula, the contributions with $k_i=0$ and $k_{i+1} >0$ are necessarily zero. Hence for $n\ge 2$ we have
\begin{equation}\label{eqfive}
p_n=\frac{(-1)^{n-1}}{x}\sum_{\kappa \in \mathcal{C}_n} \prod_{i\ge 1}\left(\frac{1}{(x+i-1)(x+i)}\right)^{k_i}\binom{k_i+k_{i+1}-1}{k_{i+1}},
\end{equation}
with $\mathcal{C}_n$ the set of sequences $\kappa=(k_1,k_2,\ldots,k_l)$ of $l\le n-1$ (strictly) positive integers summing to $n-1$. There are $2^{n-2}$ such sequences (compositions of $n-1$). For instance
\[p_1=\frac{1}{x}, \quad p_2=\frac{-1}{x^2(x+1)}, \quad p_3=\frac{1}{x^2(x+1)^2(x+2)}+\frac{1}{x^3(x+1)^2}.\]

To each $\kappa=(k_1,k_2,\ldots,k_l)\in \mathcal{C}_n$ let us associate two elements of $\mathcal{C}_{n+1}$ defined by $\kappa_1=(k_1,k_2,\ldots,k_l+1)$ and $\kappa_2=(k_1,k_2,\ldots,k_l,1)$. We have
\[\mathcal{C}_{n+1}=\bigcup_{\kappa \in \mathcal{C}_n} \{\kappa_1,\kappa_2\}.\]
Denoting $g_{\kappa},g_{\kappa_1},g_{\kappa_2}$ the respective contributions to \eqref{eqfive} of $\kappa,\kappa_1,\kappa_2$, we have easily
\[-(x+n-1)(x+n)g_{\kappa_i}\ge g_{\kappa} \quad\quad (i=1,2),\]
the equality occuring for $\kappa=(1,1,\ldots,1)$ and $i=2$.
Summing all contributions to \eqref{eqfive}, we obtain
\[-(x+n-1)(x+n)p_{n+1}\ge 2 p_n.\] 
Since \eqref{eqfour} implies
\begin{equation}\label{eqsix}
A_n= (-1)^{n-1}2\, (2n-1)!  \, p_n\arrowvert_{x=2},
\end{equation}
substituting $x=2$, the above estimate yields $A_{n+1} > A_n$, which achieves the proof of Theorem \ref{thone}.\qed

Incidentally we have shown that the generating function of the $A_n$'s is given by
\[\sum_{m\ge 1} (-1)^{m-1}\frac{A_m}{(2m-1)!}z^{m-1}=
\cfrac{2}{2 + \cfrac{z}
{3 + \cfrac{z}{4 + \cfrac{z}{5 + {}\ddots}}}}.\]

\section{Properties of $a_n$}\label{secthree}

This section is devoted to the proof of Theorem \ref{thtwo}. The following lemma plays a crucial role.
\begin{lem}\label{lemthree}
When the complete symmetric functions are specialized to $h_n=1/((x)_nn!)$ with $x>0$, we have
\[zP(z)^2+z\frac{d}{dz}\,P(z)+xP(z)=1.\]
Equivalently for any integer $n\ge 1$, we have
\begin{equation}\label{eqseven}
(n+x)p_{n+1}=-\sum_{r=1}^np_rp_{n-r+1}.
\end{equation}
\end{lem}
\begin{proof}
In view of 
\[P(z)=\frac{{}_0F_{1}(x+1;z)}{x\,{}_0F_{1}(x;z)},\]
we have
\[P(z)^2+\frac{d}{dz}\,P(z)=
\frac{1}{x^2\,{}_0F_{1}(x;z)^2}\Big({}_0F_{1}(x+1;z)^2+\frac{x}{x+1}{}_0F_{1}(x;z){}_0F_{1}(x+2;z)-{}_0F_{1}(x+1;z)^2\Big),\]
hence
\[zP(z)^2+z\frac{d}{dz}\,P(z)+xP(z)=
\frac{z}{x(x+1)}\frac{{}_0F_{1}(x+2;z)}{{}_0F_{1}(x;z)}+\frac{{}_0F_{1}(x+1;z)}{{}_0F_{1}(x;z)}.\]
We have
\[{}_0F_{1}(x;z)-{}_0F_{1}(x+1;z)=\frac{z}{x(x+1)}{}_0F_{1}(x+2;z),\]
since
\[\frac{z^k}{k!}\left(\frac{1}{(x)_k}-\frac{1}{(x+1)_k}\right)=
\frac{z}{x(x+1)}\frac{z^{k-1}}{(k-1)!}\frac{1}{(x+2)_{k-1}}.\]
\end{proof}

Substituting $x=2$, we obtain a second (inductive) proof of Theorem \ref{thone}. 
\begin{theo}\label{ththree}
The integers $\{A_n, n\ge 2\}$ are positive, increasing and given by
\[A_{n+1}=\sum_{r=1}^n\frac{n-r+1}{n+2}\binom{2n+1}{2r-1}A_rA_{n-r+1}.\]
The numbers $\{a_n, n\ge 2\}$ are positive, increasing and given by
\[a_{n+1}=\frac{1}{2}\sum_{r=1}^n\frac{1}{n+1}\binom{n+1}{r+1}\binom{n+1}{r-1} a_ra_{n-r+1}.\]
\end{theo}
\begin{proof}
We substitute \eqref{eqsix} into \eqref{eqseven}. Then we substitute $A_n=a_nC_n/2$. In both cases we proceed by induction.
\end{proof}

The positive integer
\[c_{n,r}=\frac{2}{n+1}\binom{n+1}{r+1}\binom{n+1}{r-1}=
\binom{n}{r-1}\binom{n}{r}-\binom{n}{r-2}\binom{n}{r+1}\]
is known~\cite{G,G2} to be the number of walks of $n$ unit steps on the square lattice (i.e. either up, down, right or left) starting from (0,0), finishing at $(2r-n-1,1)$ and remaining in the upper half-plane $y\ge 0$. One has $c_{n,r}=c_{n,n-r+1}$, which corresponds to symmetry with respect to  the $y$-axis.

We need some arithmetic results about $c_{n,r}$. Given a rational number $r$, there exist unique integers $a,b,m$ with $a$ and $b$ odd such that $r=2^m a/b$. The integer $m$ is called the $2$-adic valuation of $r$ and denoted $\nu_2(r)$. A classical theorem of Kummer (appeared in 1852, see~\cite{Sur} for an inductive proof) states that the $2$-adic valuation of $\binom{a+b}{a}$ is a nonnegative integer, equal to the number of ``carries'' necessary to add $a$ and $b$ in base $2$.

\begin{lem}\label{lemfour}
(i) If $n$ is even, $c_{n,r}$ is even.

(ii) If $n$ is odd and $r$ even, $c_{n,r}$ is a multiple of $4$.
\end{lem}

\begin{proof}\textit{(i)} Since $n+1$ is odd, we have
\[\nu_2(c_{n,r})=1+\nu_2\left(\binom{n+1}{r+1}\right)+\nu_2\left(\binom{n+1}{r-1}\right)\ge 1.\]
\textit{(ii)} Since $r+1$ is odd, we have
\[\nu_2(c_{n,r})=1+\nu_2\left(\binom{n}{r}\right)+\nu_2\left(\binom{n+1}{r-1}\right).\]
Since $r-1$ and $n-r+2$ are odd, at least one carry is needed to add them in base $2$. Thus $\nu_2(\binom{n+1}{r-1})\ge 1$.
\end{proof}

We also consider
\[c_{2p-1,p}=\frac{1}{p}\binom{2p}{p+1}^2.\]

\begin{lem}\label{lemfive}
(i) If $p$ is even, $c_{2p-1,p}$ is a multiple of $8$.

(ii) If $p$ is odd and $p\neq 2^k-1$ for any $k\ge 1$, $c_{2p-1,p}$ is a multiple of $4$.

(iii) If $p=2^k-1$ for some $k\ge 1$, $c_{2p-1,p}$ is odd.
\end{lem}
\begin{proof}\textit{(i)} We have
\[c_{2p-1,p}=\frac{4p}{(p+1)^2}\binom{2p-1}{p}^2.\]
Since $p+1$ is odd, we have
\[\nu_2(c_{2p-1,p})=3+2\nu_2\left(\binom{2p-1}{p}\right)\ge 3.\]

\textit{(ii)} Since $p$ is odd, we have
\[\nu_2(c_{2p-1,p})=2\nu_2\left(\binom{2p}{p+1}\right).\]
We may write the binary representations of $p+1$ and $p-1$ as
\begin{align*}
p+1=& \,c_{k+m}c_{k+m-1}\ldots c_{k+2}10\ldots00,\\
p-1=& \,c_{k+m}c_{k+m-1}\ldots c_{k+2}01\ldots10.
\end{align*}
At least one carry is needed to add them, except if $c_{k+m}=c_{k+m-1}=\ldots= c_{k+2}=0$, i.e. $p+1=2^k$, which is excluded. Therefore $\nu_2(\binom{2p}{p+1})\ge 1$.

\textit{(iii)} When $p+1=2^k$, the binary representations of $p+1$ and $p-1$ are $10\ldots00$ and $01\ldots10$. Therefore $\nu_2(\binom{2p}{p+1})=0.$
\end{proof}

Our proof of Theorem \ref{thtwo} also needs the following result.
\begin{theo}\label{thfour}
The numbers $\{A_n, n\ge 3\}$ and $\{C_n, n\ge 3\}$ have the same parity. They are odd if and only if $n=2^k-1$ for some $k\ge 2$.
\end{theo}
\begin{proof}It is known~\cite{O,KS} that the Catalan numbers $\{C_n, n\ge 3\}$ are odd if and only if $n=2^k-1$ for some $k\ge 2$. Therefore we have only to prove the first statement.

We proceed by induction on $n$. We use \eqref{eqthree} to express $(-1)^{n}A_n+C_n$, as a signed sum of the terms
\[y_j=\binom{2n-1}{2j-1} A_j \,C_{n-j}\]
for $1\le j \le n-1$. Two cases must be considered. 

Case 1: $n=2^m$ for some $m$. Then $y_1$ and $y_{n-1}$ are obviously odd. On the other hand, the contributions $\{y_j, 2\le j\le n-2\}$ are even, because the conditions $n=2^m$, $j=2^k-1$ and $n-j=2^l-1$ cannot be simultaneously satisfied if $j\neq 1$ or $n-j\neq 1$.

Case 2: $n\neq 2^m$ for any $m$. Then all contributions $\{y_j, 1\le j\le n-1\}$ are even. Indeed if $A_j$ and $C_{n-j}$ are odd, we have $j=2^k-1$ and $n-j=2^l-1$ with $k,l\ge 2$, since $j$ and $n-j$ cannot be equal to $1$. But then $\binom{2n-1}{2j-1}$ is even because we have
\[\binom{2n-1}{2j-1}=\binom{a+b}{a},\]
with $a=2^{k+1}-3$ and $b=2^{l+1}-2$. Then the binary representations of $a$ and $b$ are $1\ldots 1101$ ($k+1$ digits) and $1\ldots 1110$ ($l+1$ digits). Since $k,l\ge 2$, at least one carry is needed to add them. Summing all contributions, $A_n$ and $C_n$ have the same parity.
\end{proof}

\begin{proof}[\textbf{Proof of Theorem \ref{thtwo}}]
We proceed by induction on $n$. Our inductive hypothesis is twofold. We assume that for any $m\le n$, the positive number $a_m$ is an integer, and that $a_m$ is odd if and only if $m=2^k-2$ for some $k\ge 2$.

If $n$ is even, say $n=2p$, $a_{n+1}$ is an integer because by Theorem 3 we have
\[a_{2p+1}=\frac{1}{2}\sum_{r=1}^p c_{2p,r} a_ra_{2p-r+1},\]
and we apply Lemma \ref{lemfour} (i). Moreover $a_{2p+1}$ is even, since $a_r$ or $a_{2p-r+1}$ is even for each $r$.

If $n$ is odd, say $n=2p+1$, we have
\begin{equation}\label{pl}
a_{2p+2}=\frac{1}{2}\sum_{r=1}^p c_{2p+1,r} a_ra_{2p-r+2}+\frac{1}{4}c_{2p+1,p+1} a_{p+1}^2.
\end{equation}
Each term of the form $c_{2p+1,r}$ or $a_ra_{2p-r+2}$ is an even integer. This results from Lemma \ref{lemfour} (ii) if $r$ is even, and from the fact that $a_r$ and $a_{2p-r+2}$ are even, if $r$ is odd. We are left to consider the term 
\[w_p=\frac{1}{4}c_{2p+1,p+1} a_{p+1}^2.\]
This is an integer whenever $a_{p+1}$ is even. If $a_{p+1}$ is odd, by 
our induction hypothesis, one has $p+1=2^k-2$ for some $k\ge 2$, and we apply Lemma \ref{lemfive} (i). Summing all contributions, $a_{2p+2}$ is an integer. 

It remains to show that $a_{2p+2}$ is odd if and only if $2p+2=2^k-2$, i.e. $p+1=2^{k-1}-1$, for some $k$. We have seen that all terms on the right hand side of \eqref{pl} are even, except possibly for $w_p$.

If $p+1$ is even, $w_p$ is even because of Lemma \ref{lemfive} (i). If $p+1$ is odd and $p+1\neq 2^k-1$ for any $k$, $w_p$ is even by Lemma \ref{lemfive} (ii) and because $a_{p+1}$ is even. Finally the only possibility for $w_p$ to be odd occurs when $p+1=2^k-1$ for some $k$. Then in view of Theorem \ref{thfour}, $A_{p+1}$ and $C_{p+1}$ are odd. Hence $a_{p+1}^2/4=(A_{p+1}/C_{p+1})^2$ is also odd. Applying Lemma \ref{lemfive} (iii), $w_p$ is odd.
\end{proof}

\section{Extension to type $B$}

Generalized Narayana numbers have been introduced in~\cite[Section 5.2]{F} in the context of the non-crossing partition lattice for the reflection group associated with a root system. Ordinary Narayana polynomials correspond to a root system of type $A$. For a root system of type $B$, generalized Narayana polynomials are defined~\cite[Example 5.8]{F} by $W_0(z)=1$ and
\[W_r(z)=\sum_{k=0}^{r} {\binom{r}{k}}^2 z^{k}.\]
For their combinatorial study we refer to~\cite{Ch2,Ch} and references therein.   

We have $W_r\eqref{eqone}=W_r$, the central binomial coefficient, since
\[ W_r=\binom{2r}{r}=\sum_{k=0}^{r} {\binom{r}{k}}^2.\]
Moreover~\cite[equation (2.1)]{Ch} the Narayana polynomial $W_r(z)$ may be expressed in terms of central binomial coefficients as
\begin{equation}\label{eqeight}
W_r(z)=\sum_{m\ge0} z^m (z+1)^{r-2m} \binom{r}{2m}W_m.
\end{equation}

Exactly as in Section \ref{sectwo} the polynomial 
\[(z+1)\,W_r(z)-W_{r+1}(z)=-2rz^{r}+\ldots\]
may be expressed, in a unique way, in terms of the monic polynomials $z^mW_{r-2m+1}(z)$ with degree $r-m+1$. As a consequence, 
we may define the real numbers $B_m(r)$ by
\begin{equation}\label{eqnine}
(z+1)\,W_r(z)-W_{r+1}(z)=\sum_{m\ge1}
(-z)^m\binom{r}{2m-1}B_m(r)\, W_{r-2m+1}(z).
\end{equation}
Note that the real numbers $\{B_m(r), m\ge 1\}$ depend on $r$ and that we have $B_1(r)=2$.

Using relation \eqref{eqeight}, and paralleling the computation in Section \ref{secone}, this definition is easily shown to be equivalent to
\[W_n=\sum_{j=1}^n (-1)^{j-1} \binom{2n-1}{2j-1} B_j(r) \,W_{n-j}.\]
Therefore $B_n:=B_n(r)$ is an integer independent of $r$, given by the recurrence formula
\begin{equation}\label{eqten}
(-1)^{n-1}B_n=W_n+\sum_{j=1}^{n-1} (-1)^{j} \binom{2n-1}{2j-1} B_j \,W_{n-j},
\end{equation}  
which can be chosen as a definition, equivalent to \eqref{eqnine}.

As in Section \ref{sectwo}, if we define
\begin{align*}
\mathbf{W}(z)&=\sum_{n\ge 0} \frac{W_n}{(2n)!}z^n=\sum_{n\ge 0} \frac{1}{(n!)^2}z^n,\\
\mathbf{B}(z)&=\sum_{m\ge 1} (-1)^{m-1}\frac{B_m}{(2m-1)!}z^{m-1}.
\end{align*}
the definition \eqref{eqten} is equivalent with
\begin{equation*}
\mathbf{B}(z)\mathbf{W}(z)=2\frac{d}{dz}\mathbf{W}(z).
\end{equation*}
Obviously $\mathbf{W}(z)={}_0F_{1}(1;z)$. Therefore we are in the situation of Lemma \ref{lemtwo} with $x=1$ and we have
\begin{equation}\label{eqel}
B_n= (-1)^{n-1}2\, (2n-1)!  \, p_n\arrowvert_{x=1}.
\end{equation} 

As a consequence, the integers $\{B_n, n\ge 1\}$ are positive and increasing. Moreover their generating function is given by
\[\sum_{m\ge 1} (-1)^{m-1}\frac{B_m}{(2m-1)!}z^{m-1}=
\cfrac{2}{1 + \cfrac{z}
{2 + \cfrac{z}{3 + \cfrac{z}{4 + {}\ddots}}}}.\]

By an easy transformation of \eqref{eqten}, it is clear that the positive numbers $b_n=B_n/W_n$ are inductively defined by
\[(-1)^{n-1}b_n=1+\sum_{k=1}^{n-1}(-1)^{k} \binom{n-1}{k-1}\binom{n}{k}b_k, \]
which proves that $b_n$ is an integer. This sequence of positive integers is not new. It was studied by Carlitz~\cite[equation \eqref{eqthree}]{Ca} in connection with Bessel functions.

Substituting \eqref{eqel} into Lemma \ref{lemthree}, we have also
\begin{align*}
B_{n+1}&=\sum_{r=1}^n\frac{n-r+1}{n+1}\binom{2n+1}{2r-1}B_rB_{n-r+1},\\
b_{n+1}&=\sum_{r=1}^n\binom{n}{r}\binom{n}{r-1}b_rb_{n-r+1}.
\end{align*}
The last equation was already mentioned in~\cite{Ca}. It is very similar to Theorem \ref{ththree}. Actually the positive integer
\[\binom{n}{r}\binom{n}{r-1}\]
is known~\cite{G,G2} to be the number of walks of $n$ unit steps on the square lattice (i.e. either up, down, right or left) starting from (0,0) and finishing at $(2r-n-1,1)$.\\

\noindent\textit{Remark: }The Narayana polynomial of type $D$ is given~\cite[Figure 5.12]{F} by $W_r(z)-rzC_{r-1}(z)$. There is an analog of \eqref{eqtwo} and (10), but it does not seem to involve an integer sequence.

\section{Open problems}

As indicated at the end of the introduction, we are in lack of a combinatorial interpretation for $a_n$, possibly in the framework of probability theory. Here are two open questions which might be also investigated.

\subsection{Schur functions}

It is known~\cite[p. 23]{Ma} that the elementary symmetric functions $e_n$ satisfy
\[ne_n=\sum_{r=1}^n (-1)^{r-1} p_r e_{n-r}.\]
Lemma \ref{lemtwo} states that, when $h_n$ is specialized to $1/((x)_nn!)$ with $x>0$, $(-1)^{n-1} p_n$ is positive. By an easy induction this implies the positivity of $e_n$.

Since $h_n=s_n$ and $e_n=s_{1^n}$ are the row and column Schur functions, we are led to the following conjecture (checked for any partition $\lambda$ with weight $\le 12$).
\begin{conj}
When the complete symmetric functions are specialized to $h_n=1/((x)_nn!)$ with $x>0$, any Schur function $s_\lambda$ takes positive values.
\end{conj}

A nice expression of $s_\lambda$ in terms of $x$ may also exist. For elementary symmetric functions $e_n$, such a formula can be obtained  as a consequence of their generating function~\cite{Wo}

\[E(z)=\sum_{n\ge0}e_nz^n=\frac{1}{{}_0F_{1}(x;-z)}=\cfrac{1}{1 - \cfrac{u_1z}
{1+\cfrac{u_2z}{1-u_2z +\cfrac{u_3z}{1-u_3z+ \cfrac{u_4z}{1-u_4z+\ddots}}}}},\]
with $u_n=1/(n(x+n-1))$. For another way of writing $E(z)$ as a continuous fraction involving ratios of Schur functions, we refer to~\cite[Sections 4.3 and 5.2]{Las}.

\subsection{An identity}

By a classical result~\cite[p.335]{K} we have 
\[C_{r+1}-2\,C_r=\frac{2}{r}\binom{2r}{r-2}.\]
For $z=1$ the definition \eqref{eqtwo} yields
\[2\,C_n-C_{n+1}=\sum_{m\ge1}
(-1)^{m}\binom{n-1}{2m-1}A_m\, C_{n-2m+1},\]
hence
\[\sum_{m\ge1}
(-1)^{m}\binom{n-1}{2m-1}A_m\, C_{n-2m+1}=-\frac{2}{n}\binom{2n}{n-2}.\]
This formula can be used in two ways to define $A_n$ inductively. Actually if we write it for $n=2r$ and $n=2r+1$, we obtain
\begin{align*}
(-1)^{r-1} A_r &=\frac{1}{r}\binom{4r}{2r-2}+\sum_{m=1}^{r-1}
(-1)^m\binom{2r-1}{2m-1}A_m\, C_{2r-2m+1},\\
(-1)^{r-1} 4r A_r &=\frac{2}{2r+1}\binom{4r+2}{2r-1}+\sum_{m=1}^{r-1}
(-1)^m\binom{2r}{2m-1}A_m\, C_{2r-2m+2}.
\end{align*}
Both definitions of $A_r$ must be consistent, which yields
\begin{multline*}
\sum_{m=1}^{r-1}
(-1)^mA_m\left(\binom{2r}{2m-1}\, C_{2r-2m+2} -4r
\binom{2r-1}{2m-1}\, C_{2r-2m+1}\right)\\=
4\binom{4r}{2r-2}-\frac{2}{2r+1}\binom{4r+2}{2r-1}.
\end{multline*}
Dividing both sides by $C_{2r}$ we obtain 
\[\sum_{m=1}^{r-1}
\Big(-\frac{1}{4}\Big)^{m}\,\frac{A_m }{(2m-1)!}\,\frac{r-m}{2r-2m+3}\prod_{i=0}^{2m-2}\frac{(2r-i)^2}{4r-2i-1}=-\frac{2r(r-1)}{(2r+1)(2r+3)},\]
a very strange identity for which we are in lack of an interpretation.

A similar identity is obtained for the $B_n$'s by using relation \eqref{eqnine} for $z=1$ and paralleling the above computation. We get
\begin{multline*}
\sum_{m=1}^{r-1}
(-1)^mB_m\left(\binom{2r}{2m-1}\, W_{2r-2m+1} -4r
\binom{2r-1}{2m-1}\, W_{2r-2m}\right)\\=
(4r+2)W_{2r}-8rW_{2r-1}-W_{2r+1}.
\end{multline*}
Dividing both sides by $W_{2r}$ this yields 
\[\sum_{m=1}^{r-1}
\Big(-\frac{1}{4}\Big)^{m}\,\frac{B_m }{(2m-1)!}\,\frac{(r-m)^2}{(2r-2m+1)^2}\prod_{i=1}^{2m-1}\frac{(2r-i)^2}{4r-2i-1}=-\frac{r-1}{2(2r+1)}.\]

In view of \eqref{eqsix} and \eqref{eqel} it would be interesting to investigate whether such an identity is a general property of the $p_n$'s.

\section*{Acknowledgements}

It is a pleasure to thank Christian Krattenthaler, Alain Lascoux, Doron Zeilberger and the referee for their remarks.

\end{document}